\documentclass[12pt]{amsart}
\usepackage[margin = 0.95in]{geometry}
\usepackage{amssymb}
\usepackage{graphicx}
\usepackage{subfig}
\begin{document}
\author{Andrey Sarantsev}
\title{IID Time Series Testing}
\date{\today. \\ University of Nevada, Reno, Department of Mathematics and Statistics.\\ 
Email: \texttt{asarantsev@unr.edu}}

\begin{abstract}
Traditional white noise testing, for example the Ljung-Box test, studies only the autocorrelation function (ACF). Time series can be heteroscedastic and therefore not i.i.d. but still white noise (that is, with zero ACF). An example of heteroscedasticity is financial time series: times of high variance (financial crises) can alternate with times of low variance (calm times). Here, absolute values of time series terms are not white noise. We could test for white noise separately for original and absolute values, for example using Ljung-Box tests for both. In this article, we create an omnibus test which combines these two tests. Moreover, we create a general framework to create various i.i.d. tests. We apply tests to simulated data, both autoregressive linear and heteroscedastic. 
\end{abstract}

\maketitle

\thispagestyle{empty}

\theoremstyle{plain}
\newtheorem{thm}{Theorem}
\newtheorem{lemma}[thm]{Lemma}
\newtheorem{prop}[thm]{Proposition}
\newtheorem{cor}[thm]{Corollary}

\theoremstyle{definition}
\newtheorem{defn}{Definition}
\newtheorem{asmp}{Assumption}

\theoremstyle{remark}
\newtheorem{rmk}{Remark}
\newtheorem{exm}{Example}

\section{Introduction}

\subsection{Classic white noise testing} Take a time series $(X_n)_{n \in \mathbb Z}$ which is {\it weakly stationary:} its {\it autocorrlelation function} $\gamma(k) := \mathbb E[X_0X_k] = \mathbb E[X_tX_{t+k}]$ does not depend on $t \in \mathbb Z$. There exist a rich theory of these models with several classes of models: autoregressions, moving average models, their generalizations \textsc{ARIMA} (autoregressive integrated moving averages), stochastic volatility, generalized autoregressive conditional heteroscedastic \textsc{GARCH}, and others. We refer readers to classic textbooks: a comprehensive monograph \cite{TSTM} and its more applied version by the same authors, \cite{Book}; and a textbook \cite{Tsay} with emphasis on finance and connections to continuous-time models. See also exposition of time series in \cite[Part V]{Econometrics} with connection to linear regression models, economics, and finance. 

The simplest model is the so-called {\it white noise}: when 
\begin{equation}
\label{eq:white}
\gamma(k) = 0,\, k = 1, 2, \ldots
\end{equation}
A stronger condition on this time series is being independent identically distributed (i.i.d.) with finite second moment. Below in~\eqref{eq:SV} and~\eqref{eq:GARCH}, we provide some examples of white noise which are not i.i.d. sequences. However, if  these random variables are jointly Gaussian, then~\eqref{eq:white} implies independence. 

Given time series data $X_1, \ldots, X_N$, we can compute an empirical version $\hat{\gamma}(k)$ of the ACF: empirical correlation between $X_1, \ldots, X_{N-k}$ and $X_{k+1}, \ldots, X_N$. Under broad conditions, this is a consistent and asymptotically normal estimate for the true ACF, \cite[Theorem 7.2.2]{TSTM}.

White noise is the building block of other time series models, for example autoregression or moving average (where innovations must be white noise), or heteroscedastic models. In classic regression models, we assume that residuals form an i.i.d. Gaussian sequence. After fitting a linear regression, we need to test for {\it serial autocorrelation:} essentially,~\eqref{eq:white} for residuals. We refer the reader to detailed discussion in \cite[Chapter 7]{Econometrics}. 

Let us describe several testing methods for white noise. The most classic approach is to use empirical ACF. Earlier, we mentioned that empirical ACF is a consistent and asymptotically normal estimate of the true ACF. In particular, in case of the white noise, we have the following convergence in distribution, see \cite[Example 2.4.2]{Book}:
\begin{equation}
\label{eq:CLT-ACF}
\hat{\gamma}(k) \to 0,\quad N^{1/2}\hat{\gamma}(k) \to \mathcal N(0, 1),\quad N \to \infty,\quad k = 1, 2, \ldots
\end{equation}
One can test visually by plotting an empirical ACF, as seen in \textsc{Figure~\ref{fig:acf-bubble} (A)}. For fixed values of $k$, one can test the value of ACF at $k$ using asymptotic normality. Finally, there exist combined (portmanteau) test for empirical ACF values with lags $k = 1, \ldots, K$ for fixed $K$. The most classic example is the {\it Box-Pierce test}, see \cite{BoxPierce}: sum of squares of empirical ACF values. It follows from~\eqref{eq:CLT-ACF} that after normalization, this sum converges weakly to the $\chi^2_k$ distribution, as $N \to \infty$. This allows us to test the white noise hypothesis. A more precise modification, the {\it Ljung-Box test}, was developed in \cite{Ljung}, see also  \cite[Section 2.2]{Tsay}. Such autocorrelation tests are good in distinguishing white noise from linear autoregressive models, for example \textsc{AR}(1) or \textsc{MA}(1):
\begin{equation}
\label{eq:ARMA}
X_t - b = a(X_{t-1} - b) + Z_t,\quad X_t = Z_t + aZ_{t-1},
\end{equation}
where $Z_t$ are i.i.d. with mean zero, and $a \in (0, 1)$. Their ACF is different from zero, and this difference can be easily captured using these empirical ACF tests. This is discussed in detail in the textbooks \cite{Book, TSTM, Tsay}.

\subsection{Heteroscedasticity} We remind the readers again that white noise is a weaker condition than i.i.d. In particular, a time series can be white noise but not i.i.d. because it is heteroscedastic. The term {\it heteroscedasticity} means time-dependent variance of time series. For example, a {\it white noise in the strong sense:} a sequence of independent identically distributed random variables with zero mean and finite second moment, is {\it homoscedastic}. This concept of white noise in the strong sense is to be distinguished from the classic white noise, or white noise in the weak sense, described above. 

Homoscedasticity means that high absolute values of past terms do not influence current terms. This can fail even when the ACF is zero. Indeed, the classic Pearson correlation tests only for linear dependence. This leaves out quadratic or other forms of dependence, in particular, dependence upon absolute values. Financial series, in constrast, exhibit heteroscedastic behavior. Turbulent periods of crises with high volatility alternate with calm periods with low volatility. In terms of time series, this corresponds to high or low $|X_k|$ or, equivalently, $X_k^2$. One example of a white noise sequence which is not i.i.d. is the following {\it stochastic volatility} model:
\begin{equation}
\label{eq:SV}
X_t = V_tZ_t,\quad \ln V_t - v = a(\ln V_{t-1} - v) + W_t, 
\end{equation}
where $a \in (0, 1)$ and $Z_t,\, W_t$ are i.i.d. with mean zero. This sequence has zero autocorrelation function values. Thus it is white noise in the sense of But large $|X_t|$ implies likely large $V_t$, then likely large $V_{t+1}$, and likely large $|X_{t+1}|$. Thus the sequence $X_1, X_2, \ldots$ is not independent identically distributed. This model was introduced in \cite{Stein} (for continuous time); see also discussion in \cite[Section 3.12]{Tsay}. Another example is the \textsc{GARCH}(1, 1), introduced in \cite{GARCH}:
\begin{equation}
\label{eq:GARCH}
X_t = V_tZ_t,\quad V_t^2 = a + bX_{t-1}^2 + cV_{t-1}^2,
\end{equation}
where $a, b, c$ are positive constants, and $Z_1, Z_2, \ldots$ are i.i.d. with mean zero. Unlike~\eqref{eq:SV}, the model~\eqref{eq:GARCH} has only one innovation white noise sequence. Its volatility term does not form a separate time series; instead, it depends on previous values $X_{t-1}$ of the observed process $X$. The model~\eqref{eq:GARCH} is the classic benchmark for heteroscedastic financial time series. See \cite[Section 3.5]{Tsay} or \cite[Section 10.3.5]{Book}. 

%For SV or GARCH parameter estimation, we assume that only $X$ is observed, but the volatility process $V$ is not observed. This strict rule can be relaxed for some financial time series. In fact, we can observe VIX (the volatility index) for Standard \& Poor 500 index. This data, available daily since 1990, shows that indeed turbulent times with high volatility are alternating with calm periods with low volatility. This supports our assertion that heteroscedasticity is intrinsic to financial data. 

\begin{figure}[t]
\subfloat[$\varepsilon(t)$]{\includegraphics[width = 8cm]{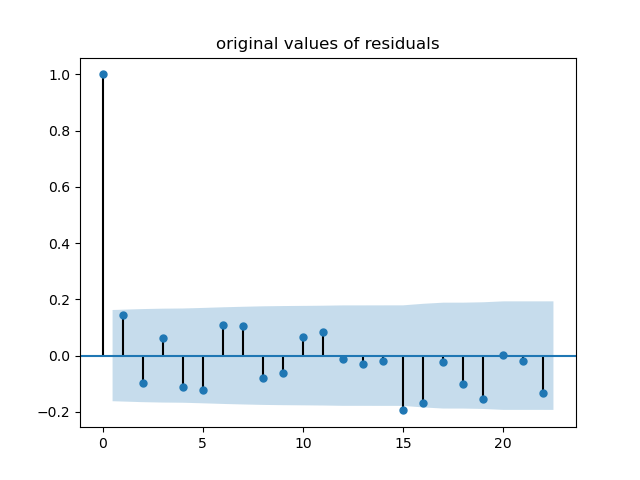}}
\subfloat[$|\varepsilon(t)|$]{\includegraphics[width = 8cm]{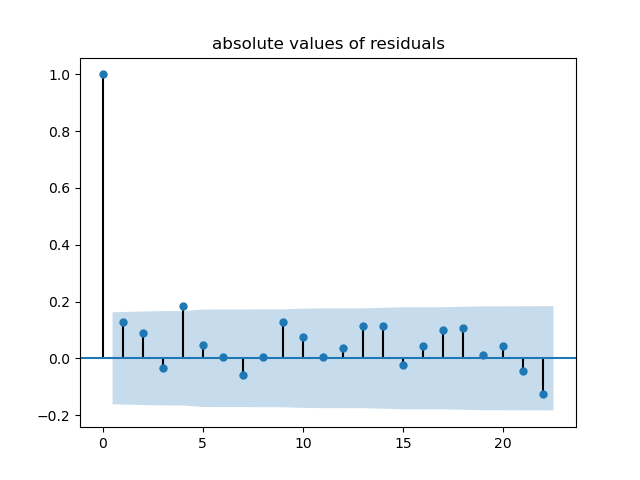}}
\caption{ACF for original and absolute values of residuals $\varepsilon(t)$ for a regression in the author's article \cite{IDY}. Is this enough to conclude that these values are i.i.d.?}
\label{fig:acf-bubble}
\end{figure}

\subsection{Main problem: i.i.d. testing} Testing whether a time series data is white noise, in the sense of zero ACF, as in~\eqref{eq:white}, is solved. A much harder problem is testing whether a time series data is a sample of independent identically distributed random variables. As we saw earlier, this property is stronger than white noise. 

Some motivation:  in~\eqref{eq:SV} and~\eqref{eq:GARCH} we do assume the innovations are i.i.d. and not just white noise. This is in contrast with linear models~\eqref{eq:ARMA}, which work well even under a weaker assumption of white noise. This is one motivation to test for i.i.d. In addition, as discussed above, regression residuals are assumed to be i.i.d. Gaussian. However, if we do not know a priori that their distribution is Gaussian, we need to test for i.i.d. not just white noise. As mentioned earlier, for Gaussian time series white noise implies independence. However, it is well-known that distributions in quantitative finance are not Gaussian: They have tails heavier than Gaussian. Thus it is essential to design testing for i.i.d. which does not depend on the assumption that the distribution is Gaussian. 

\subsection{Contributions of this article} A possible solution is a white noise test for the sequence of absolute values or squares. For example, we can plot the empirical ACF for $X^2$, or the empirical ACF for $|X|$. Together with an empirical ACF test for $X$, this will serve as a test for white noise in the strong sense. This is discussed, for example, with regard to GARCH models in \cite[Section 10.3.5]{Book}. We give another example: We analyzed financial data using linear regression in \cite[Section 3]{IDY}, and tested residuals $X$ for IID: We plot ACF for residuals and separately ACF for absolute values of these residuals; see \textsc{Figure~\ref{fig:acf-bubble} (A), (B)} for $X$ and $|X|$, respectively.

However, combining two tests in one test is not straightforward. We would like to avoid testing $X$ and $|X|$ separately. The current article is devoted to this question. 

Our goal is to make a test for i.i.d. similar to the white noise testing described above. Classic white noise tests distinguishes i.i.d. from~\eqref{eq:ARMA}, but not from~\eqref{eq:SV} or~\eqref{eq:GARCH}. The white noise test for $X_t^2$ distinguishes i.i.d. from ~\eqref{eq:SV} or~\eqref{eq:GARCH}, but it is not clear how well it distinguishes i.i.d. from~\eqref{eq:ARMA}. In this article, we create a test which separates i.i.d. from all these alternatives. 

We note that we do not know {\it a priori} the distribution of $X_i$. It can be Gaussian, which is the classic case (and often assumed in the literature). It can be another symmetric distribution, for example Laplace. Or it can be a general asymmetric distribution, such as skew-normal or asymmetric Laplace. 

We prove general theoretical results applied to any i.i.d. sequence. This gives us flexible framework for various statistical tests for i.i.d. We apply these tests to simulated four alternatives of i.i.d. series: two homoscedastic models, \textsc{AR}(1) and \textsc{MA}(1), as in~\eqref{eq:ARMA}, and two heteroscedastic models, stochastic volatility from~\eqref{eq:SV} and \textsc{GARCH}(1, 1) from~\eqref{eq:GARCH}. 

\subsection{Organization of this article} In Section 2, we state and prove our main theoretical results. In Section 3, we discuss applications of these results to construction of statistical testing. In Section 4, we apply these tests to the four alternatives mentioned above: autoregression and moving average, as in~\eqref{eq:ARMA}, stochastic volatility~\eqref{eq:SV}, and GARCH~\eqref{eq:GARCH}. We compare performance of our tests with classic Ljung-Box ACF tests for $X$ and $|X|$. The Appendix is devoted to technical definitions and lemmas. The Python code and resulting data for simulations are provided open-access in \texttt{GitHub} repository \texttt{asarantsev/IIDtest}.

\subsection{Notation} We denote by $\Rightarrow$ weak convergence. The Kronecker matrix product is denoted by $\otimes$, and the transpose of $A$ by $A'$. If random variables $X$ and $Y$ are equal in distribution, we write $X \stackrel{d}{=} Y$. The Euclidean norm of a vector $x$ is denoted by $\|x\|$. Finally, $I_k$ stands for the $k\times k$-identity matrix. 

\section{Main Theoretical Results} In this section, we state and prove our results for an i.i.d. sequence. These results will help us in \textsc{Section}~3 to construct a family of statistical tests for the following null hypothesis.

\bigskip

{\bf Null Hypothesis.} Random variables $X_1, X_2, \ldots$ are independent identically distributed.

\bigskip

Assume $X \stackrel{d}{=} X_i$. For a function $f : \mathbb R \to \mathbb R$, let
$$
m(f) := \mathbb E[f(X)],\quad Q(f) := \mathrm{Var}(f(X)),\quad s(f) := [Q(f)]^{1/2},\quad \mathcal K(f) := \mathbb E[f^4(X)].
$$
Take the first $N$ terms $f(X_1), \ldots, f(X_N)$. Compute empirical mean and standard deviation:
$$
\hat{m}(f) := \frac1N\sum\limits_{k=1}^Nf(X_k)
\quad \mbox{and} \quad 
\hat{s}^2(f) := \frac1N\sum\limits_{i=1}^N\left(f(X_i) - \hat{m}(f)\right)^2.
$$
Next, compute empirical autocovariance and autocorrelation for the lag $k$: 
$$
\hat{\gamma}_k(f) := \frac1{N-k}\sum\limits_{i=1}^{N-k}\left(f(X_i) - \hat{m}(f)\right)\left(f(X_{i+k}) - \hat{m}(f)\right)\quad \mbox{and}\quad \hat{\rho}_k(f) = \frac{\hat{\gamma}_k(f)}{\hat{s}^2(f)}.
$$
For two functions $f, g : \mathbb R \to \mathbb R$, define the covariance and correlation: 
\begin{equation}
\label{eq:cov-corr}
Q(f, g) := \mathrm{Cov}(f(X), g(X)),\quad C(f, g) := \frac{Q(f, g)}{s(f)s(g)}.
\end{equation}
If $\mathcal K(f) < \infty$ and $\mathcal K(g) < \infty$, then $Q(f, g) < \infty$. We can compute empirical cross-covariance and cross-correlation of $f(X_1), \ldots, f(X_N)$ and $g(X_1), \ldots, g(X_N)$ with time lag $k$:
\begin{equation}
\label{eq:corr}
\hat{\gamma}_k(f, g) := \frac1{N-k}\sum\limits_{i=1}^{N-k}(f(X_i) - \hat{m}(f))(g(X_{i+k}) - \hat{m}(g))\quad \mbox{and}\quad \hat{\rho}_k(f, g) := \frac{\hat{\gamma}_k(f, g)}{\hat{s}(f)\hat{s}(g)}.
\end{equation}
For the case $k = 0$ (ordinary covariance and correlation of $f(X)$ and $g(X)$, without any time lag), we skip the subscript $k = 0$ in~\eqref{eq:corr}. 

\begin{asmp}
\label{asmp:main}
Take functions $f_1, \ldots, f_m : \mathbb R \to \mathbb R$ which satisfy $\mathcal K(f_i) < \infty$. 
\end{asmp}

\begin{thm} Under the null hypothesis and Assumption~\ref{asmp:main}, the sequence of $K$ vectors in $\mathbb R^{m^2}$
\begin{equation}
\label{eq:cov-vector}
\hat{\mathbb{Q}}_k := \Bigl[\hat{\gamma}_k(f_i, f_j),\, i, j = 1, \ldots, m\Bigr],\quad k = 1, \ldots, K.
\end{equation}
of empirical auto- and cross-covariances from~\eqref{eq:corr} satisfies the Central Limit Theorem:
\begin{equation}
\label{eq:CLT-cov}
N^{1/2}\left[\hat{\mathbb{Q}}_1, \ldots, \hat{\mathbb{Q}}_K\right] \Rightarrow  (\Xi_1, \ldots, \Xi_K),\quad \Xi_k \sim N_{m^2}(\mathbf{0}, \mathbf{Q})\quad \mbox{i.i.d.}
\end{equation}
where the limiting covariance matrix $\mathbf{Q}$ is given by 
\begin{equation}
\label{eq:cov-matrix}
(\mathbf{Q})_{(ij), (i'j')} := Q(f_i, f_{i'})Q(f_j, f_{j'}).
\end{equation}
\label{thm:main-cov}
\end{thm}

\begin{proof} We apply \cite[Theorem 1]{ACF}. Let us translate the notation: 
\begin{equation}
\label{eq:indices}
\ell := k',\, \alpha = i',\, \beta = j'.
\end{equation}
The multivariate white noise $\mathbf{X}_t := \begin{bmatrix} f_1(X_t) & \ldots & f_m(X_t) \end{bmatrix}$ (assuming without loss of generality that $m(f_i) = 0$ for all $i$, then $\mathbb E\mathbf{X}_t = \mathbf{0}$) has covariance matrix $V := (Q(f_i, f_j))_{i, j = 1, \ldots, m}$. Next, we compute limiting mean vector and covariance matrix. From \cite[(2.1)]{ACF}, asymptotic mean of $\hat{\gamma}_k(f_i, f_j)$ is zero, since $\delta_k = 0$ for $k = 1, \ldots, K$. To compute the covariance matrix, we need the concept of a {\it cumulant} of four random variables from Definition~\ref{defn:cum}. From Lemmas~\ref{lemma:cumulant} and~\ref{lemma:independent} in the Appendix, the cumulant of the following random variables is zero:
\begin{equation}
\label{eq:four}
f_i(X_t),\, f_j(X_{t+k}),\, f_{i'}(X_{t+v}),\, f_{j'}(X_{t+k'+v}).
\end{equation}
It is denoted in \cite[(2.2)]{ACF} by $K_{iji'j'}(0, k, v, k'+v)$, allowing for this different notation~\eqref{eq:indices}. Take $k \ne k'$. Then in \cite[(2.2)]{ACF}, we get $\delta_{k' - k} = \delta_{k' + k} = 0$. Thus the limiting covariance is $0$. For $k = k'$, the limiting covariance is nonzero, and we shall compute it: In \cite[(2.2)]{ACF} we get: $\delta_{k' - k} = 1$ but $\delta_{k' + k} = 0$. Thus the limiting covariance is $Q(f_i, f_{i'})Q(f_j, f_{j'})$. Since we are interested only in asymptotics, we can remove $N-k$ from the denominator in \cite[(2.2)]{ACF}, and put $N^{1/2}$ in the left-hand side. 
\end{proof}

\begin{thm} Under the null hypothesis and Assumption~\ref{asmp:main}, the sequence of $K$ vectors in $\mathbb R^{m^2}$
\begin{equation}
\label{eq:corr-vector}
\hat{\mathbb{C}}_k := \Bigl[\hat{\rho}_k(f_i, f_j),\, i, j = 1, \ldots, m\Bigr],\quad  k = 1, \ldots, K.
\end{equation}
of empirical auto- and cross-correlations from~\eqref{eq:corr} satisfies a similar Central Limit Theorem:
\begin{equation}
\label{eq:CLT-corr}
N^{1/2}\left[\hat{\mathbb{C}}_1, \ldots, \hat{\mathbb{C}}_K\right] \Rightarrow \mathcal (\tilde{\Xi}_1, \ldots, \tilde{\Xi}_K),\quad \tilde{\Xi}_k \sim N_{m^2}(\mathbf{0}, \mathbf{C})\quad \mbox{i.i.d.}
\end{equation}
with the limiting covariance matrix $\mathbf{C}$ given by 
\begin{equation}
\label{eq:corr-matrix}
(\mathbf{C})_{(ij), (i'j')} := C(f_i, f_{i'})C(f_j, f_{j'}).
\end{equation}
\label{thm:main-corr}
\end{thm}

\begin{proof} Let us show~\eqref{eq:CLT-corr}. By the Law of Large Numbers, $\hat{s}(f) \to s(f)$ as $N \to \infty$ in probability. Apply Slutsky's theorem: \cite[Section 5.5]{Textbook}. From~\eqref{eq:corr} and~\eqref{eq:CLT-cov}, we have convergence in law:
\begin{align}
\label{eq:Slutsky}
\begin{split}
&N^{1/2}\Bigl[\hat{\rho}_k(f_i, f_j),\, i, j = 1, \ldots, m,\, k = 1, \ldots, K\Bigr] \\  \Rightarrow \Bigl[\eta_{ijk} &:= \frac{\xi_{ijk}}{s(f_i)s(f_j)},\, i, j = 1, \ldots, m,\, k = 1, \ldots, K\Bigr],\\ [\xi_{ijk},\, & i, j = 1, \ldots, m] \sim \mathcal N_{m^2}(\mathbf{0}, \mathbf{Q}),\quad k = 1, \ldots, K.
\end{split}
\end{align}
From~\eqref{eq:Slutsky}, we get: $[\eta_{ijk}] \sim \mathcal N_{m^2}(\mathbf{0}, \mathbf{C})$ i.i.d. for $k = 1, \ldots, K$, with
\begin{align*}
(\mathbf{C})_{ij, i'j'} &:= \frac{\mathbf{Q}_{ij, i'j'}}{s(f_i)s(f_j)s(f_{i'})s(f_{j'})} = \frac{Q(f_i, f_{i'})Q(f_{j}, f_{j'})}{s(f_i)s(f_j)s(f_{i'})s(f_{j'})} \\ & = C(f_i, f_{i'})C(f_j, f_{j'}).
\end{align*}
In the last line, we used~\eqref{eq:cov-corr}. This completes the proof of~\eqref{eq:CLT-corr}. 
\end{proof}

\begin{rmk}
In terms of Kronecker matrix product, see \cite{Kronecker}, we can write the limiting covariance matrices in~\eqref{eq:cov-matrix} and~\eqref{eq:corr-matrix} as $\mathbf{Q} = \mathcal Q\otimes\mathcal Q$, where $\mathcal Q = (Q(f_i, f_j))_{i, j = 1, \ldots, m}$ is the covariance matrix of the random vector $(f_1(X), \ldots, f_m(X))$, and similarly $\mathbf{C} = \mathcal C\otimes\mathcal C$, where $\mathcal C = (C(f_i, f_j))$ is the correlation matrix of this random vector. 
\label{rmk:Kronecker} 
\end{rmk}

\begin{lemma} Under the null hypothesis and Assumption~\ref{asmp:main}, if correlations are zero:
\begin{equation}
\label{eq:zero-cov}
Q(f_i, f_j) = \mathrm{Cov}(f_i(X), f_j(X)) = 0\quad \mbox{for}\quad i, j = 1, \ldots, m,\quad i \ne j,
\end{equation}
then the empirical vectors from~\eqref{eq:corr} have the following weak limit:
\begin{equation}
\label{eq:corr-result}
N^{1/2}\left[\hat{\mathbb{C}}_1, \ldots, \hat{\mathbb{C}}_K\right] \Rightarrow N_{m^2K}(\mathbf{0}, I_{m^2K}).
\end{equation}
\label{lemma:iid}
\end{lemma}

\begin{proof} Immediately follows from Theorem~\ref{thm:main-corr}, since $C(f_i, f_j) = 0$ for $i \ne j$ by~\eqref{eq:zero-cov}, and $C(f_i, f_i) = 1$ (correlation of a random variable $f_i(X)$ with itself is one). 
\end{proof}

This convergence to a vector of i.i.d. standard Gaussians enables us to design statistical tests, see \textsc{Section}~3. However, if the correlations are nonzero, then we can still have convergence to this vector. Only we need to multiply the empirical correlation vector by a certain constant matrix. Take the set $\mathcal P(n)$ of positive definite symmetric $n\times n$-matrices. For every $A \in \mathcal P(n)$, we can find an $n\times n$-matrix $B$ (not necessarily symmetric or positive definite) such that $BB' = A^{-1}$. We can find a version of this matrix $B$ which continuously depends on $A$. In other words, take a continuous function $\Phi : \mathcal P(n) \to \mathbb R^{n\times n}$ such that $\Phi(A)\Phi'(A) = A^{-1}$. One example is the inverse of the matrix square root $\Phi(A) = (A^{1/2})^{-1} = (A^{-1})^{1/2}$.

\begin{lemma} Under the null hypothesis and Assumption~\ref{asmp:main}, the normalized vector of empirical auto- and cross-covariances satisfies, as $N \to \infty$,
\begin{equation}
\label{eq:CLT-cov-iid}
N^{1/2}\cdot\Phi(\mathcal Q)\otimes\Phi(\mathcal Q)\left[\hat{\mathbb Q}_1, \ldots, \hat{\mathbb Q}_K\right] \Rightarrow \mathcal N_{m^2K}(\mathbf{0}, I_{m^2K}).
\end{equation}
Similarly, the vector of empirical auto- and cross-correlations satisfies, as $N \to \infty$,
\begin{equation}
\label{eq:CLT-corr-iid}
N^{1/2}\cdot\Phi(\mathcal C)\otimes\Phi(\mathcal C)\left[\hat{\mathbb C}_1, \ldots ,\hat{\mathbb C}_K\right] \Rightarrow \mathcal N_{m^2K}(\mathbf{0}, I_{m^2K}).
\end{equation}
\label{lemma:CLT-iid}
\end{lemma}

\begin{proof}
The Kronecker matrix product commutes with ordinary matrix product, and therefore with matrix inversion; see the monograph \cite{Kronecker}. Thus the $m^2\times m^2$-matrix $\mathbf{E} := \Phi(\mathcal Q)\otimes\Phi(\mathcal Q)$ satisfies $(\mathbf{E}\mathbf{E}')^{-1} = \mathcal Q\otimes\mathcal Q = \mathbf{Q}$. From the properties of the multivariate normal distribution and~\eqref{eq:CLT-cov}, we conclude~\eqref{eq:CLT-cov-iid}. Similarly, we get~\eqref{eq:CLT-corr-iid}. 
\end{proof}

However, sometimes we do not know a priori the limiting covariance matrices $\mathbf{Q}$ or $\mathbf{C}$, since we do not know the covariance and correlation functions $Q(f_i, f_j)$ and $C(f_i, f_j)$. We must estimate all these from the data. Therefore, we need to state a version of Lemma~\ref{lemma:CLT-iid} with empirical correlations instead of theoretical ones. Define 
$$
\hat{\mathcal Q} := \left(\hat{\gamma}(f_i, f_j)\right)_{i, j = 1, \ldots, m}\quad \mbox{and}\quad \hat{\mathcal C} := \left(\hat{\rho}(f_i, f_j)\right)_{i, j = 1, \ldots, m}. 
$$

\begin{thm} Under the null hypothesis and Assumption~\ref{asmp:main}, replace in~\eqref{eq:CLT-cov-iid} the limiting covariance matrix $\mathcal Q$ with its estimate $\hat{\mathcal Q}$. Then, as $N \to \infty$,
\begin{equation}
\label{eq:estimated-cov}
N^{1/2}\Phi(\hat{\mathcal Q})\otimes\Phi(\hat{\mathcal Q})\left[\hat{\mathbb Q}_1, \ldots, \hat{\mathbb Q}_K\right] \Rightarrow \mathcal N_{m^2K}(\mathbf{0}, I_{m^2K}).
\end{equation}
Similarly, replace in~\eqref{eq:CLT-cov-iid} the limiting covariance matrix $\mathcal C$ with its estimate $\hat{\mathcal C}$. As $N \to \infty$,
\begin{equation}
\label{eq:estimated-corr}
N^{1/2}\Phi(\hat{\mathcal C})\otimes\Phi(\hat{\mathcal C})\left[\hat{\mathbb C}_1, \ldots, \hat{\mathbb C}_K\right] \Rightarrow \mathcal N_{m^2K}(\mathbf{0}, I_{m^2K}).
\end{equation}
\label{thm:empirical}
\end{thm}

\begin{proof} By consistency of the covariance estimates, $\hat{\mathcal Q} \to \mathcal Q$ as $N \to \infty$ in probability (actually, even almost surely). Since the function $\Phi$ is continuous, we get: $\Phi(\hat{\mathcal Q} \to \Phi(\mathcal Q)$ and therefore $\Phi(\hat{\mathcal Q})\otimes\Phi(\hat{\mathcal Q}) \to \Phi(\mathcal Q)\otimes\Phi(\mathcal Q)$. Applying Slutsky's theorem \cite[Section 5.5]{Textbook} for random vectors to~\eqref{eq:CLT-cov}, we get:
\begin{align*}
N^{1/2}\Phi(\hat{\mathcal Q})\otimes\Phi(\hat{\mathcal Q})\left[\hat{\mathbb Q}_1, \ldots, \hat{\mathbb Q}_K\right] \Rightarrow  \Phi(\mathcal Q)\otimes\Phi(\mathcal Q)\left[Y_1, \ldots, Y_K\right] \stackrel{d}{=} \left[Z_1, \ldots, Z_K\right].
\end{align*}
for i.i.d. $Y_k \sim \mathcal N_{m^2}(\mathbf{0}, \mathbf{Q})$ and i.i.d. $Z_k \sim \mathcal N_{m^2}(\mathbf{0}, \mathbf{Q})$. This completes the proof of~\eqref{eq:estimated-cov}. The proof of~\eqref{eq:estimated-corr} is similar.  
\end{proof}

\section{Statistical Tests}

Theoretical results from the previous section allow us to design statistical tests. 

\subsection{The case of uncorrelated functions} Most obviously, Lemma~\ref{lemma:iid} has the following practical value. 

\begin{cor}
Under assumptions of Lemma~\ref{lemma:iid}, a function of auto- and cross-correlations:
\begin{equation}
\label{eq:statistic}
\mathbb{T}_N := N\sum\limits_{k=1}^K\|\hat{\mathbb C}_k\|^2 \equiv N\sum\limits_{i=1}^m\sum\limits_{j=1}^m\sum\limits_{k=1}^K\left[\hat{\rho}_k(f_i, f_j)\right]^2
\end{equation}
satisfies the following weak convergence result, as $N \to \infty$:
\begin{equation}
\label{eq:chi-squared}
\mathbb{T}_N \Rightarrow \chi^2_{m^2K}.
\end{equation}
\label{cor:iid}
\end{cor}

\begin{proof}
This follows immediately from Lemma~\ref{lemma:iid}: The sequence of $K$ vectors $\hat{\mathbb C}_k \in \mathbb R^{m^2}$ for $k = 1, \ldots, K$ forms a vector $\hat{\mathbb C}$ from $\mathbb R^{m^2K}$ with Euclidean norm 
$$
\|\hat{\mathbb C}\|^2 = \sum\limits_{k=1}^K\|\hat{\mathbb C}_k\|^2 = \sum\limits_{i=1}^m\sum\limits_{j=1}^m\sum\limits_{k=1}^K\left[\hat{\rho}_k(f_i, f_j)\right]^2.
$$
From~\eqref{eq:corr-result}, we get: $N^{1/2}\hat{\mathbb C} \Rightarrow \mathcal N_{m^2K}(\mathbf{0}, I_{m^2K})$. The  sum of squares of $n$ i.i.d. random variables is distributed as the $\chi^2$ random variable with $n$ degrees of freedom. Apply this to $n = m^2K$.  Using continuity of the square of the Euclidean norm as a function of its vector, we conclude~\eqref{eq:chi-squared}.
\end{proof}

Corollary~\ref{cor:iid} allows us to devise the following test: Fix significance level $p$ (for example, 5\%), and take $u$ such that \begin{equation}
\label{eq:chi2}
W \sim \chi^2_{m^2K},\quad \mbox{has}\quad \mathbb P(W > u) = p.
\end{equation}
We reject the null hypothesis with significance level $p$ if $\mathbb{T}_N > u$. This would be an analogue of the Box-Pierce test from~\cite{BoxPierce}. In fact, we get the Box-Pierce test if we take $m = 1$ and $f_1(x) = x$ in Assumption~\ref{asmp:main}. In the \textsc{Introduction,} we mentioned that there exist a better (more precise) version of this test, which is called Ljung-Box test, \cite{Ljung}. 

\begin{cor}
Under assumptions of Lemma~\ref{lemma:iid}, fix a constant $c > 0$ and take the function
\begin{equation}
\label{eq:statistics-Ljung}
\mathbb{L}_N := N(N+c)\sum\limits_{k=1}^K\frac{\|\hat{\mathbb C}_k\|^2}{N-k}.
\end{equation}
Then we have the following weak convergence result, as $N \to \infty$:
\begin{equation}
\label{eq:Ljung-conv}
\mathbb{L}_N \Rightarrow \chi^2_{m^2K}. 
\end{equation}
\label{cor:iid-Ljung}
\end{cor}

\begin{proof}
Obviously, as $N \to \infty$, for each fixed $k$, we have: 
\begin{equation}
\label{eq:conv}
a(N, k) := \left[\frac{N(N+c)}{N-k}\right]^{1/2} \to 1,\quad N \to \infty.
\end{equation}
Multiply the $k$th component $\hat{\mathbb C}_k$ of the vector $\hat{\mathbb C}$ by $a(N, k)$. Apply Slutsky's theorem, \cite[Section 5.5]{Textbook}, and use~\eqref{eq:conv}. Then we get:
\begin{equation}
\label{eq:new-conv}
\left[a(N, 1)\hat{\mathbb C}_1, \ldots, a(N, k)\hat{\mathbb C}_K\right] \Rightarrow \mathcal N_{m^2K}(0, I_{m^2K}).
\end{equation}
Taking the square of the Euclidean norm on both sides of~\eqref{eq:new-conv}, similarly to the proof of Corollary~\ref{cor:iid}, we get~\eqref{eq:Ljung-conv}. 
\end{proof}

Again, similarly to Corollary~\ref{cor:iid}, the result of Corollary~\ref{cor:iid-Ljung} allows us to create a statistical test: Reject the null hypothesis with significance level $p$, if $\mathbb{L}_N > u$, where $u$ is taken from~\eqref{eq:chi2}. The Ljung-Box test from \cite{Ljung} is a particular case of this test, if we take $m = 1$, $f_1(x) = x$, and $c = 2$. We impose the following reasonable assumptions on the distribution of $X$.

\begin{asmp} The distribution of $X$ is symmetric: $X \stackrel{d}{=} -X$.
\label{asmp:symmetry}
\end{asmp}

Next, we assume that the tails of each $X_i$ are not very heavy. 

\begin{asmp} The distribution of $X$ has finite fourth moment: $\mathbb E[X^4] = 0$. 
\label{asmp:moment-4}
\end{asmp}

\begin{rmk}
\label{rmk:simple-case}
Under Assumptions~\ref{asmp:symmetry} and~\ref{asmp:moment-4}, we can take $f_1(x) = x$ and $f_2(x) = |x|$. Indeed, then $Q(f_1, f_2) = \mathbb E[X|X|] = 0$. This family of two functions satisfies Assumption~\ref{asmp:main} and it has zero correlation.
\end{rmk}

However, if only Assumption~\ref{asmp:symmetry} holds, but Assumption~\ref{asmp:moment-4} does not hold, then we can easily modify these functions and make them bounded: $f_1(x) = \sin (ax)$ and $f_2(x) = \cos (ax)$, or any other two bounded functions, one odd and the other even. 

\subsection{The general case} Next, assume we found a sequence of functions satisfying Assumption~\ref{asmp:main}, but they do not have zero correlation. Then Lemma~\ref{lemma:iid} is not applicable. Instead, we apply Theorem~\ref{thm:empirical} and get the following results. 

\begin{cor} Under the null hypothesis and Assumption~\ref{asmp:main}, define the matrix 
$\hat{\mathbf C} := \hat{\mathcal C}\otimes\hat{\mathcal C}$. Next, define the following functions:
\begin{align}
\label{eq:functions-general}
\begin{split}
\tilde{\mathbb{T}}_N &:= N\cdot \sum\limits_{k=1}^K\hat{\mathbb C}'_k\hat{\mathbf C}^{-1}\hat{\mathbb C}_k,
\\ \tilde{\mathbb{L}}_N &:=  N(N+c)\cdot \sum\limits_{k=1}^K\frac{1}{N-k}\hat{\mathbb C}'_k\hat{\mathbf C}^{-1}\hat{\mathbb C}_k.
\end{split}
\end{align}
These statistics satisfy the following limit theorem:
$$
\tilde{\mathbb{T}}_N \Rightarrow \chi^2_{m^2K};\quad \tilde{\mathbb{L}}_N \Rightarrow \chi^2_{m^2K},\quad N \to \infty.
$$
The same results hold true if we replace in~\eqref{eq:functions-general} all $\hat{\mathbb C}_k$ by $\hat{\mathbb Q}_k$, and all $\hat{\mathcal C}$ by $\hat{\mathcal Q}$. 
\label{cor:test-general}
\end{cor}

\begin{proof}
We make changes in the proofs of Corollaries~\ref{cor:iid} and~\ref{cor:iid-Ljung}. Designate $\hat{\mathbf{E}}_N := \Phi(\hat{\mathcal C})\otimes\Phi(\hat{\mathcal C})$. Instead of the square norm of the $m^2$-dimensional vector $\hat{\mathbb C}_k$, we have the square norm of the $m^2$-dimensional vector $\hat{\mathbf{E}}_N\hat{\mathbb C}_k$. As discussed in the proof of Theorem~\ref{thm:empirical}, $\hat{\mathbf{E}}_N\hat{\mathbf{E}}_N' = \hat{\mathbb C}^{-1}$. Thus 
\begin{align*}
\left\|\hat{\mathbf{E}}_N\hat{\mathbb C}_k\right\|^2 = \left[\hat{\mathbf{E}}_N\hat{\mathbb C}_k\right]'\cdot\left[\hat{\mathbf{E}}_N\hat{\mathbb C}_k\right]  = \hat{\mathbb C}_k'\hat{\mathbf{E}}_N\hat{\mathbf{E}}_N'\hat{\mathbb C}_k = \hat{\mathbb C}_k' \hat{\mathbb C}^{-1}\hat{\mathbb C}_k.
\end{align*}
\end{proof}

This corollary allows us to create statistical tests, similarly to the case of uncorrelated functions $f_1, \ldots, f_m$ from the previous subsection. Take a significance level $p$, find the cutoff $u$ from~\eqref{eq:chi2}, and reject the mull hypothesis with significance level $p$ if $\tilde{\mathbb{T}}_N > u$ (analogue of the Box-Pierce test from \cite{BoxPierce}), or 
$\tilde{\mathbb{L}}_N > u$ (analogue of the Ljung-Box test from \cite{Ljung}). 

\section{Simulation} 

\subsection{Time series models} We applied this test to four models:

\begin{itemize}
\item \textsc{AR}(1) (autoregression of order $1$);

\item \textsc{MA}(1) (moving average of order $1$); 

\item \textsc{GARCH}(1, 1) (generalized conditional heteroscedastic autoregression of order $1$);

\item \textsc{SV}: a stochastic volatility model with log volatility modeled as \textsc{AR}(1), with innovations independent of the innovations for observed process. Unlike other three models, it has two independent series of innovations.
\end{itemize}

We compare our new test with Ljung-Box test (for both original and absolute values), for 5 lags. These models are not white noise. Thus our tests should reject the null white noise hypothesis. But existing Ljung-Box tests should reject it too. Which test is better? We can judge by the $p$-values. If one test has lower $p$-values that another test (when the null hypothesis is false), then the first test is better. We cannot expect our new test to improve upon the existing Ljung-Box tests, because our test is {\it portmanteau:} It combines white noise testing for original values and for absolute values.  However, we still compare the $p$-values to see which test performs better in which model. 

In each of four models, we consider both Gaussian and Laplace innovations: i.i.d. $(Z_t)$, with sample size $N = 100$. Each distribution has mean $\mathbb E[Z_t] = 0$ and variance $\mathrm{Var}(Z_t) = 1$, with densities $f(z) := (2\pi)^{-1/2}\exp(-z^2/2)$ and $f(z) = 0.5\exp(-|z|)$. Indeed, much of time series theory is (explicity or implicitly) based on the assumption that innovations is Gaussian. Thus we chose another symmetric distribution, with tails heavier than Gaussian. 

Note that these distributions for innovations are symmetric and have finite fourth moment. In other words, they satisfy Assumptions~\ref{asmp:symmetry} and~\ref{asmp:moment-4}. The same applies to the time series $(X_t)$. Therefore, we can use two test functions from Remark~\ref{rmk:simple-case}. 

In addition, a small $a$ in autoregression and moving average models means that this time series model is close to the white noise. A large $a$ (close to $1$) means that our time series model is very different from the white noise. The same is true for the stochastic volatility model (for $a$ from the autoregression for log volatility), and for \textsc{GARCH} (where the role of $a$ is played by the sum of parameters; see below). We take several values of parameters $a$: $0.1, 0.2, 0.3, 0.4, 0.5$.

\subsection{Results} We provide $p$-values in the four tables on the next page. For \textsc{AR} and \textsc{MA}, in most cases the Ljung-Box test for original values works better (that is, gives lower $p$-values) than the new test. But the Ljung-Box test for absolute values mostly fails to reject the null hypothesis; thus, it does not work as well as the new test. This seems to be a feature of linear time series models. 

Indeed, deviations from white noise in linear models are manifested via the autocorrelation function. For these models, white noise is equivalent to i.i.d. although of course for general time series models this is not the case. The classic Ljung-Box test deals directly with this function. Thus we do not need additional testing of autocorrelation functions for some related time series (such as absolute values) to reject the null hypothesis (i.i.d.)

For \textsc{GARCH} with Gaussian innovations, our results are inconclusive. For \textsc{GARCH} with Laplace innovations, almost in all cases the new test is better than both Ljung-Box tests. Finally, for \textsc{SV} with Gaussian innovations, in all cases the new test is better than both its competitors. And for \textsc{SV} with Laplace innovations, the new test is worse acrosss the board. Thus the results are inconclusive. We did not come up with a theoretical reason. 

\subsection{Further research} Further simulation testing is required to establish whether these results are statistical illusions. Testing for simulated \textsc{ARMA}$(p, q)$ models would be very valuable. Same is true for heteroscedastic models: We studied only one particular stochastic volatility model; there exist many versions. As for \textsc{GARCH}, we studied only \textsc{GARCH}$(p, q)$ with $p = q = 1$ all with parameters $a, b, c$ from~\eqref{eq:GARCH} equal to each other. 

Lemma~\ref{lemma:iid} and Theorem~\ref{thm:empirical} allow us to create other statistical tests for the null hypothesis. For example, we can take the maximum norm instead of the $L^2$-norm of the vector $\hat{\mathbb C}$ or other vectors. For the classic ACF, this was done in \cite{Max}. More generally, we can take the $L^p$-norm for $p \in [1, \infty]$, with $p = \infty$ corresponding to the maximum norm. 

\newpage

\begin{table}[t]
\begin{tabular}{||c||c|c|c||c|c|c||}
\hline
\hline
$a$ & $p_{O, G}$ & $p_{A, G}$ & $p_{N, G}$  & $p_{O, L}$ & $p_{A, L}$ & $p_{N, L}$\\
\hline
\hline
$0.1$ & $0.194$ & $	0.385$ & $0.453$ & $0.047$ & $0.057$ & $0.167$ \\
\hline
$0.2$ & $0.045$ & $	0.484$ & $0.071$ & $0.003$ & $0.052$ & $0.023$ \\
\hline
$0.3$ & $0.003$ & $	0.452$ & $0.129$ & $0.002$ & $0.005$ & $0.000$ \\
\hline
$0.4$ & $0.002$ & $0.413$ & $0.016$ & $0.000$ & $0.002$ & $0.000$ \\
\hline
$0.5$ & $0.000$ & $	0.246$ & $0.001$ & $0.000$ & $0.011$ & $0.000$ \\
\hline
\hline
\end{tabular}

\bigskip

\caption{Autoregression results: $p$-values. Fix an $a$ and simulate a sequence $(X_t)$ with $X_t = aX_{t-1} + Z_t$, $X_0 = 0$.}

\begin{tabular}{||c||c|c|c||c|c|c||}
\hline
\hline
$a$ & $p_{O, G}$ & $p_{A, G}$ & $p_{N, G}$  & $p_{O, L}$ & $p_{A, L}$ & $p_{N, L}$\\
\hline
\hline
$0.1$ & $0.095$ & $0.382$ & $0.417$ & $0.164$ & $0.292$ & $0.109$ \\
\hline
$0.2$ & $0.013$ & $	0.379$ & $0.330$ & $0.018$ & $0.382$ & $0.223$ \\
\hline
$0.3$ & $0.021$ & $	0.288$ & $0.017$ & $0.004$ & $0.496$ & $0.109$ \\
\hline
$0.4$ & $0.003$ & $	0.303$ & $0.005$ & $0.002$ & $0.067$ & $0.033$ \\
\hline
$0.5$ & $0.001$ & $	0.325$ & $0.005$ & $0.000$ & $0.240$ & $0.016$ \\
\hline
\hline
\end{tabular}

\bigskip

\caption{Moving average results: $p$-values. Fix an $a$  and simulate a sequence $(X_t)$ with $X_t = Z_t + aZ_{t-1}$.}

\begin{tabular}{||c||c|c|c||c|c|c||}
\hline
\hline
$a$ & $p_{O, G}$ & $p_{A, G}$ & $p_{N, G}$  & $p_{O, L}$ & $p_{A, L}$ & $p_{N, L}$\\
\hline
\hline
$0.1$ & 0.487 & 0.535 & 0.210 & 0.050 & 0.071	 & 0.036 \\
\hline
$0.2$ & 0.342 & 0.432 & 0.085 & 0.088 & 0.126	 & 0.094 \\
\hline
$0.3$ & 0.201 & 0.517 & 0.022 & 0.143 & 0.090	 & 0.235 \\
\hline
$0.4$ & 0.183 & 0.216 & 0.164 & 0.212 & 0.074	 & 0.514 \\
\hline
$0.5$ & 0.369 & 0.242 & 0.215 & 0.298 & 0.055 & 0.342 \\ 
\hline
\hline
\end{tabular}

\bigskip

\caption{Stochastic volatility results: $p$-values. We apply the tests to $X_t = e^{V_t}Z_t$, where $V_t = aV_{t-1} + W_t$, with independent $W_t \sim \mathcal N(0, 1)$.}

\begin{tabular}{||c||c|c|c||c|c|c||}
\hline
\hline
$a$ & $p_{O, G}$ & $p_{A, G}$ & $p_{N, G}$  & $p_{O, L}$ & $p_{A, L}$ & $p_{N, L}$\\
\hline
\hline
0.1	& 0.015 & 0.103 & 0.329 & 0.192 & 0.155 & 0.143\\
\hline
0.2	& 0.093 & 0.093 & 0.082 & 0.282 & 0.206 & 0.142\\
\hline
0.3 & 0.067 & 0.020 & 0.110 & 0.393 & 0.214 & 0.093\\ 
\hline
0.4	& 0.088	& 0.067	& 0.088 & 0.028 & 0.143 & 0.013\\
\hline
0.5	& 0.063 & 0.014 & 0.080 & 0.216 & 0.057 & 0.095\\
\hline
\hline
\end{tabular}
\bigskip
\caption{Stochastic volatility results: $p$-values. Here we apply the tests to $X_t = V_tZ_t$, where $V_t^2 = (1 + V_{t-1}^2 + X_{t-1}^2)\cdot (a/3)$. Thus we take all three parameters of GARCH to be equal to each other. Their sum is equal to the parameter $a$ used for other models.}
\end{table}

Indices O, A, N correspond to Ljung-Box for original values, Ljung-Box for absolute values, and our new test. Letters G and L correspond to Gaussian and Laplace innovations.

\newpage

\section*{Appendix}

\begin{defn} 
The {\it cumulant} of four random variables $X, Y, Z, W$ is defined as
$$
\mathbb E[X'Y'Z'W'] - \mathbb E[X'Y']\cdot \mathbb E[Z'W'] - \mathbb E[X'W']\cdot\mathbb E[Y'Z'] - \mathbb E[X'Z']\cdot\mathbb E[Y'W'],
$$
where $X', Y', Z', W'$ are centered versions of random variables $X, Y, Z, W$.
\label{defn:cum}
\end{defn}

\begin{lemma} 
Assume we can arrange four random variables $X, Y, Z, W$ into two independent bivariate random vectors. Then the cumulant of these four random variables is zero.
\label{lemma:cumulant}
\end{lemma}

\begin{proof} 
Assume, without loss of generality, that $(X, Y)$ is independent of $(Z, W)$. Then the corresponding centered vectors $(X', Y')$ and $(Z', W')$ are also independent. Thus 
\begin{align*}
\mathbb E[X'Y'Z'W'] &= \mathbb E[X'Y']\cdot \mathbb E[Z'W'],\\
\mathbb E[X'Z'] &= \mathbb E[X']\cdot \mathbb E[Z'] = 0 \cdot 0 = 0,\\
\mathbb E[X'W'] &= \mathbb E[X']\cdot\mathbb E[W'] = 0 \cdot 0 = 0.
\end{align*}
Combining these observations, we complete the proof. 
\end{proof}

\begin{lemma}
\label{lemma:independent}
Under conditions of Theorem~\ref{thm:main-cov}, for any $t, k, k', v \in \mathbb Z$ such that $t, k, k' \ge 1$, we can arrange random variables from~\eqref{eq:four} into two independent bivariate random vectors. 
\end{lemma}

\begin{proof}
Since $X_u,\, u = 1, 2, \ldots$ are i.i.d. we need only to show that we can arrange the four integers $a_0 := t$, $a_1 := t + k$, $a_2 := t + v$, $a_3 := t + v + k'$ using a bijection $\rho$ of $\{0, 1, 2, 3\}$ onto itself so that $\{a_{\rho(0)}, a_{\rho(1)}\} \cap \{a_{\rho(2)}, a_{\rho(3)}\} = \varnothing$. We consider two cases: (a) $v \ne k$ and $v \ne -k'$: then $\{t, t + v\} \cap \{t + k, t + v + k'\} = \varnothing$; (b) $v = k$ or $v = -k'$; then $\{t , t + v + k'\} \cap \{t + k, t + v\} =  \varnothing$.
\end{proof}

\bibliographystyle{plain}

\end{document}